\documentclass[11pt]{amsart}
\usepackage{amssymb,amsmath,latexsym,enumerate, dsfont}
\usepackage{graphicx,verbatim,enumerate,%
ifpdf,mdwlist}
\usepackage{color}
\usepackage{mathabx}
\ifpdf
\usepackage{hyperref}%[unicode]{hyperref}
\else
\usepackage[hypertex]{hyperref}
\fi
\usepackage{graphicx}
\usepackage{cancel}
\usepackage{float}
\renewcommand{\phi}{\varphi}

\newtheorem{theorem}{Theorem}[section]

\newtheorem{lemma}[theorem]{Lemma}

\newtheorem{proposition}[theorem]{Proposition}

\newtheorem{defi}[theorem]{Definition}

\newtheorem{exa}[theorem]{Example}

\newtheorem{rem}[theorem]{Remark}
\newtheorem{claim}[theorem]{Claim}

\newcommand{\IsoT}{\operatorname{IsoType}}

\newcommand{\A}{\mathcal{A}}

\newcommand{\set}[1]{\{{#1}\}}
\newcommand{\rel}[1]{\mathrel{#1}}

\title[Isomorphism types of equivalence structures]{Comparing the isomorphism types of  equivalence structures and preorders}

\author[N.~Bazhenov]{Nikolay Bazhenov}
\address{Sobolev Institute of Mathematics, pr. Akad. Koptyuga 4, Novosibirsk,
630090 Russia;\
Novosibirsk State University, ul. Pirogova 2, Novosibirsk, 630090 Russia
}
\email{\href{mailto:bazhenov@math.nsc.ru}{bazhenov@math.nsc.ru}}

\author[L.~San Mauro]{Luca San Mauro}
\address{Institute of Discrete Mathematics and Geometry, Vienna University of Technology, Vienna, Austria}
\email{\href{mailto:luca.san.mauro@tuwien.ac.at}{luca.san.mauro@tuwien.ac.at}}

\thanks{
Bazhenov was supported by the program of fundamental scientific researches of the SB~RAS No.~I.1.1, project No.~0314-2019-0002.
San Mauro was supported by the Austrian Science Fund FWF, project~M~2461}



\begin{document}
\begin{abstract}
A general theme of computable structure theory is to investigate when structures have copies of a given complexity $\Gamma$. We discuss such problem for the case of equivalence structures and preorders. We show that there is a $\Pi^0_1$ equivalence structure with no $\Sigma^0_1$ copy, and in fact that the isomorphism types realized by the $\Pi^0_1$ equivalence structures coincide with those realized by the $\Delta^0_2$ equivalence structures. We also construct a $\Sigma^0_1$ preorder with no $\Pi^0_1$ copy.
\end{abstract}

\maketitle

\section{Introduction}
A primary question of computable structure theory is to ask,  given a familiar class of countable structures $\mathfrak{A}$,  which of the isomorphism types of $\mathfrak{A}$ can be realized by computable structures. Sometimes the answer is trivial:
for instance, \emph{all} countable vector spaces over $\mathbb{Q}$ have a computable copy~\cite{nurtazin1974strong}; on the other hand,  Peano Arithmetic has only \emph{one} computable model~\cite{tennenbaum1959non}. In other cases, such as abelian $p$-groups,  one (partially)  answers in terms of a class of invariants~\cite{khisamiev1998constructive}. But many natural classes (e.g., linear orders, Boolean algebras, or the class of all graphs) are simply too rich to hope for a nice characterization of which members of the class have a computable copy.

An obvious way of generalizing the above problem is to investigate when structures have copies of some complexity $\Gamma$~--- where $\Gamma$ is  some level of the arithmetical, hyperarithmetical, analytical, or Ershov hierarchy~--- and then studying the relation between the isomorphism types realized by structures of different complexity. 
%Such a relation often turns out to be interesting. In particular, 
Interestingly, it can be the case that all the isomorphism types realized by structures of a certain complexity are already realized by structures of considerably less complexity: e.g., Spector famously proved that every hyperarithmetic ordinal has a computable copy~\cite{spector1955recursive}. In this paper, we contribute to this research thread by focusing on two case-studies: equivalence structures   and preorders.

Equivalence structures are among the simplest structures which are nontrivial. For this reason, they form a class  which is often  reasonably tame but still worth examining. In fact, despite of their structural simplicity, equivalence structures have also remarkably deep effective properties: %Let us give a couple of examples.
 Calvert, Cenzer, Harizanov, and Morozov~\cite{calvert2006effective} proved that the index set of the computable equivalence structures isomorphic to a given one may be $\Pi^0_4$-complete; Downey, Melnikov, and Ng showed~\cite{downey2017friedberg} that there is an effective list of all computable equivalence structures, thus classifying them in a sense proposed by Goncharov and Knight~\cite{goncharov2002computable}. 
  Cenzer, Harizanov, and Remmel~\cite{cenzer2011sigma10} started the study of the isomorphism types of $\Sigma^0_1$ and $\Pi^0_1$ equivalence structures; our work here extends their study.

Another piece of motivation for focusing on equivalence structures comes from the fact that measuring the complexity of equivalence relations with domain $\omega$ has been a longstanding endeavour in the literature. In particular, researchers investigated for decades the effective properties of computably enumerable equivalence relations (called \emph{ceers}, after Gao and Gerdes~\cite{gao2001computably}) and they compared them via a natural effectivization of Borel reducibility called \emph{computable reducibility}: see, e.g., 	\cite{andrews2014universal, coskey2012hierarchy, andrews2018joins,ianovski2014complexity, fokina2019measuring}. In comparison with ceers, the study of equivalence relations of higher complexity has been less extensive. Yet, computable reducibility has been used to analyze equivalence relations of various complexity, including some that are not even hyperarithmetical  (such as 
%the minimal equivalence relations of the analytic hierarchy~\cite{bazhenov2019minimal}, and 
the isomorphism relations for familiar classes of computable structures~\cite{fokina2012isomorphism}).
% or such as the equivalence relations which are \emph{minimal} in the analytic hierarchy that we constructed in \cite{}). 
In recent times, following the  work of Ng and Yu~\cite{ng2019degree},  we initiated  a systematic study of $\Delta^0_2$ equivalence relations: we proved that theory of ceers, co-ceers, and $\Delta^0_2$ equivalence relations behave quite differently~\cite{bazhenov2018classifying}. This motivates  the following question (from which the present paper originated):
\begin{center}
\emph{Are ceers and co-ceers distinguishable in terms of the isomorphism types that they realize?}
\end{center} 

Theorem \ref{thm:coceervsceer} below positively answers to such question. In fact, we give a full description of the comparison between the isomorphism types realized at different levels of the arithmetical hierarchy: in particular, we show that for a non-zero $n$, the isomorphism types of $\Delta^0_{n+1}$ equivalence structures are precisely the same as those of $\Pi^0_n$ equivalence structures (Theorems~\ref{theo:pi01-delta02} and~\ref{theo:arithmetical}). Furthermore, we separate different levels of the analytical hierarchy via the corresponding realizable types of equivalence structures (Proposition~\ref{prop:analytical}).
%%%--NEW
 
Our second case-study is also inspired by the research on computable reducibility. Preorders occur naturally in logic (e.g., think of the structure induced by $\vdash_{PA}$ on the formulas of arithmetic) and, in fact, their  complexity  has been  studied since the 1980s~\cite{montagna1985universal} (see \cite{andrews2019effective} for more recent results). By Theorem \ref{thm:sigma1vspi1preord} below we show that the case of preorders differs from that of the equivalence structures.

%  More generally, preorders are sufficiently similar to equivalence structure. This makes them  a nice class for try to understand  concerning the latter are some special feature of the equivalence structures, or if they reflect a more general phenomenon. 
%
% represent a natural class on which 
%
%
%
%
%The last motivation applies also to our second case-study. Preorders are one of the most natural class to investigate, if one wants to move from equivalence structures to something more general. In fact, their complexity has been studied since the 1980s the most natural step  complexity of preorders has been studies since the 1980's. 

\subsection{Preliminaries on equivalence structures and preorders} 

Equivalence structures (resp.\ preorders) are structures of the form $\A= \langle \omega, R \rangle$ where $R$ is an equivalence relation (a preorder, i.e., a reflexive and transitive relation); $\A\in \Sigma^{i}_{\alpha}$, for $i\in\set{-1,0,1}$, if $R \in \Sigma^{i}_{\alpha}$. Note that, for our interests, assuming that all structures have domain $\omega$ is not a limitation: every isomorphism type of an equivalence structure with a countably infinite domain is realized by an equivalence structure with domain $\omega$. 

We denote by $\mathbb{EQ}$   and by $\mathbb{PO}$ the class of all countable equivalence structures and all countable preorders respectively. Moreover, for any class of structures $\mathbb{K}$, we denote by $\IsoT(\mathbb{K}, \Gamma)$ the collection of all isomorphism types realized by the structures from $\mathbb{K}$ of complexity $\Gamma$.

We denote by $\set{R_e}_{e\in\omega}$ some fixed uniform enumeration of all ceers. We say that an equivalence class $[x]_{R}$ is \emph{older} than an equivalence class $[y]_{R}$ if $\min [x]_R < \min [y]_R$. This relation obviously induces a linear ordering on the $R$-classes.
%Let $\set{R_e}_{e\in\omega}$ be a uniform enumeration of all ceers.
%We say that an equivalence class $[x]_{R}$ is \emph{older} than an equivalence class $[y]_{R}$ if $\min [x]_R < \min [y]_R$. This obviously induces a linear ordering on the $R$-classes.

\subsection{Warning about our terminology}
There is a well-established way of thinking of a \emph{$\Sigma^0_1$} (resp.\ \emph{$\Pi^0_1$}) \emph{presentation} of structure $\A$ as the combination of a  pre-structure $\A^*$ and a ceer (co-ceer) $R$ such that the quotient structure $\A^* /R $ is isomorphic to $\A$. This approach generalizes the way in which many algebraic structures are presented (e.g., the homomorphism theorem says that every countable group is a quotient of the free group on
countably many generators). The study of quotient presentations dates back to the very beginning of computable model theory (see, e.g., the work of Metakides and Nerode
on c.e.\ vector spaces~\cite{metakides1977recursively}) and it gave rise  to a rich research program; the interested reader is referred to \cite{selivanov2003positive,gavruskin2014graphs,fokina2016linear}. However, it should be clear that in this paper we do \emph{not} deal with quotients, but rather with the (pre-)structure themselves. We emphasize this point with one example. Let $R$ be a preorder on $\omega$. One can regard at the algebraic properties of $R$ in two ways:
\begin{itemize}
	\item[(a)] First, considering the quotient structure $\A: = \langle R / supp(R), \leq_R\rangle$, where 	
	\begin{itemize}
	\item $supp(R) := \{ (x,y)\,\colon (x\rel{R}y) \& (y\rel{R} x) \}$;
		\item the domain of $\A$ consists of the equivalence classes $[x]_{supp(R)}$, $x\in \omega$;
		\item and $[x]_{supp(R)} \leq_R [y]_{supp(R)}$ if and only if $(x\rel{R} y)$.
	\end{itemize} 
	\item[(b)] Secondly, considering directly the structure $\langle \omega, R\rangle$.
\end{itemize}

Feiner's celebrated result~\cite{Feiner} that there is \emph{$\Sigma^0_1$-presented} linear ordering which is not isomorphic to any computable one shall be understood within the first setting. On the contrary, in the present paper we will always work according to the second setting. For example, if we say that $R$ is a \emph{$\Sigma^0_1$ linear order}, then this means the following:
\begin{itemize}
	\item $R\in \Sigma^0_1$, and 
	\item the ordering $R$ is linear (in particular, $R$ is an antisymmetric relation).
\end{itemize} Clearly, these conventions imply that in our setting, \emph{every} $\Sigma^0_1$ linear order is computable.

\subsection{Preliminaries on limitwise monotonicity}

The system of invariants which characterize the isomorphism types of computable equivalence structures can be nicely represented in terms of limitwise monotonic functions. 

Let $\mathbf{d}$ be a Turing degree. A function $F\colon \omega \to \omega \cup\{ \infty\}$ is \emph{$\mathbf{d}$-limitwise monotonic} if there is a total $\mathbf{d}$-computable function $f(x,s)$ such that
\begin{itemize}
	\item[(a)] $f(x,s) \leq f(x,s+1)$ for all $x$ and $s$,
	
	\item[(b)] $F(x) = \lim_s f(x,s)$ for all $x$.
\end{itemize}
When one is talking about $\mathbf{0}$-limitwise monotonic functions, the prefix ``$\mathbf{0}$-'' is typically omitted.

 A set $A\subseteq \omega$ is \emph{limitwise monotonic} if either $A = \emptyset$ or there is a limitwise monotonic function $F(x)$ with $range(F) = A$.

\begin{theorem}[Calvert, Cenzer, Harizanov, and Morozov~{\cite{calvert2006effective}}] \label{theo:cchm}
	Let $\mathcal{A}$ be a countable equivalence structure with infinitely many classes. Then the following conditions are equivalent:
	\begin{itemize}
		\item[(a)] $\mathcal{A}$ has a computable copy.
		
		\item[(b)] There is a limitwise monotonic function $F\colon \omega\to \omega\cup \{\infty\}$ such that for any $\kappa\in\omega\cup\{\infty\}$, the structure $\mathcal{A}$ has precisely $card(\{  x\in \omega\,\colon F(x)=\kappa\})$ classes of size $\kappa$.
	\end{itemize}
\end{theorem}

We conclude this brief discussion about limitwise monotonicity by stating two results that we will use later. The reader is referred to, e.g., \cite{DKT-11} for further results on limitwise monotonic sets and functions.

\begin{theorem}[Harris~{\cite[Lemma~5.2]{Harris}} and, independently, Kach~{\cite[Theorem~1.9]{Kach}}] \label{theo:liminf}
	A function $F\colon \omega\to \omega\cup\{ \infty\}$ is $\mathbf{0}'$-limitwise monotonic if and only if there exists a computable function $g\colon \omega\times\omega\to\omega$ such that
	\[
		F(x) = \lim\inf\!_s\, g(x,s) \text{ for all } x.
	\]
\end{theorem}

\begin{theorem}[Khoussainov, Nies, and Shore {\cite[Lemma~2.6]{KNS-97}}; see also Theorem~2.2 in~\cite{DKT-11}]
\label{theo:KNS}
	There is a d.c.e. set $A\subset\omega$, which is not limitwise monotonic.
\end{theorem}

\section{Isomorphism types of equivalence structures}

Cenzer, Harizanov, and Remmel~\cite{cenzer2011sigma10} proved that there is a $\Sigma
^0_1$ equivalence structure with no computable copy. It is then natural to  ask whether the $\Sigma
^0_1$ and the $\Pi^0_1$ equivalence structures realize the same isomorphism types. The next theorem shows that this is not the case: $\Pi^0_1$ equivalence structures are, in a sense, more expressive that $\Sigma^0_1$ equivalence structures. 

\begin{theorem}\label{thm:coceervsceer}
There is $\Pi^0_1$ equivalence structure $\A$ which is not isomorphic to any $\Sigma^0_1$ equivalence structure. That is, $\IsoT(\mathbb{EQ}, \Pi^0_1) \not\subseteq \IsoT(\mathbb{EQ}, \Sigma^0_1)$.
\end{theorem}

\begin{proof}
To prove the theorem, it is enough to build a co-ceer that it is not isomorphic to any ceer. To do so, we  construct in stages a co-ceer $S$ which, for all $e$, satisfies the following requirements:

\begin{multline*}
  P_{e}:\  \text{there is an $S$-class of size $e+1$ if and only if}\\ 
   \text{there are no $R_e$-classes of such size.}
\end{multline*}

 This obviously guarantees that $S$  cannot be isomorphic to any of the $R_e$'s.

%\begin{align*}
%  &P_{e}: \text{$\A$ realizes $u_e$ if and only if $R_e$ does not realize it.}\\
%\end{align*}
%
% diagonalizing against the isomorphism types of all ceers $\set{R_e}_{e\in\omega}$.

%\subsection*{The strategy}
%%For each $R_e$, we will build $\A$ such that $R_e$ has an equivalence class of size $u_e$ if and only $\A$ has not an equivalence class of size $n_e$.
%%
%%Each 
%
%
% We  say that two equivalence relations $R$ and $S$ \emph{agree on some $n$} if either they both have an equivalence class of size $n$ or they both do not.

\subsection*{The construction}

In the construction  we will define the set $Y_e$ of  \emph{$e$-witnesses}; at any stage, the number of $e$-witnesses will be either $e+1$ or $e +2$. The role of such $e$-witnesses will be that of realizing an equivalence class (i.e., $[\langle e,0\rangle]_{S}$) which will eventually ensure that $S\not\cong R_e$.  During the construction we will often say that we \emph{transform some $z$ into a $S$-singleton}; by this we mean that we let $y \cancel{S} z$ and $z \cancel{S} y$, for all $z \neq y$. Finally, we say that $P_e$ \emph{turns \textbf{on}} at some stage $s$ if the oldest $R_e[s]$-class of size $e+1$ differs, for all $t< s$, from the oldest $R_e[t]$-class of the same size.

\subsubsection*{Stage $0$}

For all $e$, let $Y_e[0]: = \set{x : 0< x \leq e+1}$ and let all the requirements $P_e$ be \textbf{off}. Moreover, let $S[0]$ be the following computable partition of $\omega$ in infinitely many classes:
\[
\langle a,b\rangle S[0] \langle c,d\rangle \Leftrightarrow a=c.
\]

%\noindent Furthemore,  call $Y_e=\set{x : x < n_e}$  the set of all numbers that are \emph{$e$-preserved}.

\subsubsection*{Stage $s+1=\langle e, n \rangle$}

We focus on  $P_e$. First, let 
\[
u_{e}:=\min \set{ x : x \in Y_e[s] \mbox{ and } x> e+1}
\]
 (if such set is not empty), and let
\[ 
 v_{e}:= \min \set{x: x > \max Y_e[s]\mbox{ and } \langle e,0\rangle S[s]\langle e, x\rangle}.
 \]
 . Now, we distinguish four cases:

\begin{enumerate}
\item $card(Y_e[s])=e+1$, $R_e[s+1]$ has some class of size $e+1$, and $P_e$ is \textbf{off}:\\ If so, we want to declare a new $e$-witness. We do so by setting $Y_e[s+1]:=Y_e[s] \cup \set{v_{e}}$. We also transform $\langle e, v_{e}+1 \rangle$ into a $S$-singleton;
\item $card(Y_e[s])=e+2$, $R_e[t]$ has no classes of size $e+1$, and $P_e$ is \textbf{off}: \\ If so, we want  to decrease the number of $e$-witnesses by $1$. We do so by setting  $Y_e[s+1]:=Y_e[s] \smallsetminus \set{u_{e}}$
  and transforming $\langle e, u_{e}\rangle$ into a $S$-singleton; 
\item %$|Y_e[s]|=e+2$ and 
If 
$P_e$ is \textbf{on}:\\
If so, we want to change the current set of $e$-witnesses. We do so by setting $Y_e[s+1]:=(Y_e[s]\smallsetminus \set{u_{e}})\cup \set{v_{e}}$ and transforming $\langle e, u_{e}\rangle$ into an $S$-singleton. We also turn $P_e$ \textbf{off};
\item Any other case: \\ We transform $\langle e, v_{e}\rangle$ into a $S$-singleton.
%$R_e[s+1]$ has no class of size $n_e$ and $|Y_e[s]|\neq n_e$: If so, 
\end{enumerate}

\smallskip

Finally, let $Y_e := \lim_{s\rightarrow \infty} Y_e[s]$  and let $S:=\bigcap_{s\in\omega} S[s]$.

%Let $u_e$ be the maximum of $Y_e$, and let $v_e$ be the least number which is not  $e$-preserved and $\langle e,0\rangle E_{\A}[s] \langle e,v_e\rangle$. 
%We say that $R_e$ \emph{requires attention} if, for some $t \geq \langle e, n-1\rangle$, $R_e[t]$ and $R_e[s+1]$ disagree on $n_e$. We 
%%Let $x$ and $y$ be the least two numbers such that $\langle e,0\rangle R \langle e,x\rangle R \langle e, y\rangle$ and $y$ is not $e$-preserved. 
%distinguish three cases:
%
%\begin{enumerate}
%%Let $z$ be the least number such that $\langle e,z\rangle E_{\A} \langle e,0\rangle$ and  $\langle e,z\rangle$ is not $e$-preserved.
%\item If  $R_e[s+1]$ has no classes of size $n_e$: Put $v_e$ in $Y_e$ and transform $\langle e, v_e+1\rangle$ into a $\A$-singleton;
%\item If $R_e$ requires attention and $R_e[s+1]$ has some class of size $n_e$: Transform $\langle e, u_e\rangle $ into a $\A$-singleton;
%\item If $R_e$ does not requires attention: Transform $\langle e, v_e\rangle$ into a $\A$-singleton.
%
%
%% $R_e[e,n-1]$ and $R_e[e,n]$ disagree on $n$, and $R_e[e,n]$ has no class of size $u_e$: Make all the $e$-preserved numbers as singletons, i.e., for all $x \leq z < u_e$, let $\langle e,z\rangle \cancel{R} \langle e,w\rangle$ and $\langle e,w\rangle \cancel{R} \langle e,z\rangle$;
%%\item If $R_e[e,n-1]$ and $R_e[e,n]$ disagree on $n$, and $R_e[e,n]$ has some class of size $u_e$: Observe that in this case, $x=y$. Define all the numbers $x\leq z< u_e$ as $e$-preserved.
%\end{enumerate}

\subsection*{The verification} By construction, $S$ is obviously a co-ceer. Moreover, by induction it is not difficult to see that, for all $e$, $Y_e[0]\subseteq Y_e$ and the following identity holds,
\begin{equation}
\tag{$\star$}
card([\langle e,0\rangle]_{S})=  card(Y_e).
\label{eqn:tag}
\end{equation}  The rest of the verification is based on the following lemma.

%\begin{lemma}\label{lemma:e-classes}
%For all $e$, $|[\langle e,0\rangle]_{S}|=  | Y_e|$. 
%\end{lemma}
%
%
%\begin{proof}
%We show that  $[\langle e,0\rangle]_{S}=  \set{\langle e,x\rangle : x \in Y_e}$. Assume that $z \in [\langle e,0\rangle]_{S}$. Since $S$ is a co-ceer and in the first stage we split all numbers which differ in their first coordinates, there must be some $i$ such that  $z=\langle e,i \rangle$.
%\end{proof}

\begin{lemma}
For all $e$, $S \not\cong R_e$. 
\end{lemma}

\begin{proof}
Suppose first that $R_e$ has an equivalence class of size $e+1$. So, there exists a stage $s_0$ such that any   $R_e[s]$, with $s \geq s_0$,  realizes the oldest  equivalence class of size $e+1$ of $R_e$. Let $t$ be the least number such that $\langle e, t\rangle \geq s_0$. It is not difficult to see that, since $s_0$ is chosen minimal, at  stage $\langle e,t \rangle$ we enter in case $(1)$ of the construction, by which we set the number of $e$-witnesses to $e+2$. Next, observe that in all further stages in which we focus on $P_e$ we always enter in case $(4)$ (case $(1)$ is excluded because $card(Y_e)\neq e+1$; case $(2)$ is excluded because $R_e$ has some class of size $e+1$; case $(3)$  is excluded because $R[s_0]$ already realized the oldest $R$-class of size $e+1$ and therefore $P_e$ stays \texttt{off}). This means that  we will not modify the set of $e$-witnesses.  Therefore, we have that $Y_e= Y_e[\langle e, t\rangle]$ and, by equation (\ref{eqn:tag}), this means that $card([\langle e, 0 \rangle]_{S})\neq e+1$. 

\smallskip

On the other hand, suppose that $R_e$ has no equivalence classes of size $e+1$. This can happen in two cases that shall be considered separately. 

First, assume that there is a least stage $s_0$ such that any $R_e[s]$, with $s\geq s_0$, has no equivalence classes of size $e+1$. If so, by reasoning as above, we obtain that there is a least stage $\langle e,t\rangle$ after which we have $e+1$ many $e$-witnesses and, in all further stages focusing on $P_e$, we enter in case $(4)$. Therefore, $card([\langle e, 0\rangle])=card(Y_e)=e+1$, while $R_e$ has no equivalence classes of size $e+1$.

Secondly, assume that there are infinitely many stages at which $R_e[s]$ has an equivalence class of size $e+1$, even though there is no $R_e$-class of such size. This can happen only if the approximation to $R_e$ has infinitely many mindchanges regarding the oldest equivalence class of size $e+1$. That is, only if $P_e$ turns \textbf{on} infinitely often, and therefore the construction enters in case $(3)$ infinitely often when dealing with $P_e$. We claim that  this makes $Y_e=Y_e[0]$, and therefore $card(Y_e)=e+1$. To see this, suppose that there is a least $z$ such that $z \in Y_e \smallsetminus Y_e[0]$. This means that there is stage $s_0$ such that, for all $s\geq s_0$, $z\in Y_e[s]$. Yet, it is not hard to see that there must be a stage $\langle e,n\rangle\geq s_0$ in which $P_e$ is \textbf{on} and $u_e$ takes value $z$. Hence, in this stage the construction enters in case $(3)$  which implies that  $z\notin Y_e[s+1]$, a contradiction. So, $card(Y_e)=e+1$. By equation \ref{eqn:tag}, it follows that $card([\langle e, 0 \rangle]_{S})= e+1$, while $R_e$ has no equivalence class of such size.
\end{proof}

This concludes the proof of Theorem~\ref{thm:coceervsceer}.
\end{proof}

Having shown that there is a $\Pi^0_1$  equivalence structure which has no $\Sigma^0_1$ copy, it comes natural to investigate the relation between the $\Pi^0_1$ and the $\Delta^0_2$ equivalence structures. Obviously, every $\Pi^0_1$ (or $\Sigma^0_1$) equivalence structure is isomorphic to a $\Delta^0_2$ equivalence structure. One might conjecture that the inclusion is strict, i.e., that there is a $\Delta^0_2$ equivalence structure  which is not isomorphic to any $\Pi^0_1$ equivalence structure. This is not the case.

\begin{theorem}\label{theo:pi01-delta02}
The isomorphism types of $\Pi^0_1$ and $\Delta^0_2$ equivalence structures coincide. That is, $\IsoT(\mathbb{EQ}, \Pi^0_1) = \IsoT(\mathbb{EQ}, \Delta^0_2)$.
\end{theorem}

\begin{proof}
	Let $E$ be a $\Delta^0_2$ equivalence relation. If $E$ has only finitely many equivalence classes, then it is obvious that $E$ is isomorphic to a computable structure. Hence, without loss of generality, we assume that $E$ has infinitely many classes.
	
	A relativized version of Theorem~\ref{theo:cchm} implies that there is a $\mathbf{0}'$-limit\-wise monotonic function $H\colon \omega\to\omega \cup \{\infty\}$ with the following property: For any $\alpha\in\omega\cup \{ \infty\}$, the number of $E$-equivalence classes of size $\alpha$ is equal to 
	\[
		\gamma(\alpha) := card(\{ k\in\omega\,\colon H(k) = \alpha\}).
	\]
	Note that $H(k)$ cannot be equal to zero.
	
	By Theorem~\ref{theo:liminf}, one can choose a computable function $g\colon \omega^2\to\omega$ such that for any $k\in\omega$, we have
	\[
		H(k) = \lim\inf\!_s\, g(k,s).
	\]
	Since $H(k) \neq 0$, we may assume that $g(k,s) \geq 1$ for all $s$ (i.e. if $g(k,s)$ equals zero, then one just replaces this zero with one).
	
	We give a construction of a $\Pi^0_1$ equivalence relation $R$. As per usual, we define $R$ by enumerating its complement, i.e. at any stage of the construction, we are allowed to declare that $(x\rel{\cancel{R}}y)$, where $x\neq y$. Surely, if $(x\rel{\cancel{R}}y)$, then we always (implicitly) assume $(y\rel{\cancel{R}}x)$. In other words, we can only provide negative information about $R$, and we never give positive pieces of $R$-data.
	
	At a non-zero stage $s$, we define a finite set $A[s]$ and a function $\ell[s]$ acting from $A[s]$ onto $\{ 0,1,\dots,s-1\}$. The intuition behind $\ell[s]$ is as follows~--- at the stage $s$, we think that the first $s$ classes of $R[s]$ are precisely the following sets:
	\[
		\{ x\in A[s]\,\colon \ell(x)[s] = k \},\ k=0,1,\dots,s-1.
	\]
	Thus, we will assume that we ``automatically'' declare the following $R$-in\-for\-ma\-tion: if $\ell(x)[s]\downarrow\ \neq \ell(y)[s]\downarrow$, then $(x\rel{\cancel{R}}y)$. 
	
	At a stage $s$, an element $x$ is called \emph{fresh} if $x\not\in \bigcup_{t< s} A[s]$.
	
	%\ 
	
	\subsection*{The construction}
	
	\emph{Stage $0$.} Set $A[0] := \emptyset$.
	
	\emph{Stage $s+1$.} First, choose an element $w$ as follows:
	\begin{itemize}
		\item If the set $\{ x \text{ is non-fresh}\,\colon x\not\in A[s]\}$ is non-empty, then $w$ is the least element of this set.
		
		\item Otherwise, $w$ is the least fresh element.
	\end{itemize}
	Declare $\ell(w)[s+1] := s$.
	
	After that, for each $k<s$, we act according to the following instructions. We define
	\[
		\delta(k) := card(\{ x\in A[s]\,\colon \ell(x)[s] = k\}).
	\]
	
	\begin{enumerate}
		\item If $g(k,s+1) > \delta(k)$, then choose the least $(g(k,s+1)-\delta(k))$ fresh elements. For each of these elements $y$, set $\ell(y)[s+1] := k$.
		
		\item If $g(k,s+1) < \delta(k)$, then find the greatest $(\delta(k) - g(k,s+1))$ elements $z$ such that $\ell(z)[s] = k$. For each of these $z$, set $\ell(z)[s+1]$ undefined. Surely, we also assume $z\not\in A[s+1]$.
	\end{enumerate}
	
	The description of the construction is finished.
	
	\subsection*{The verification} 
	
	It is clear that the constructed relation $R$ is co-c.e. Moreover, it is not hard to verify the following claim (recall that $g(k,s) \geq 1$ for all $k$ and $s$).
	
	\begin{claim}
		Every $x\in\omega$ satisfies precisely one of the following two cases:
		\begin{itemize}
			\item[(a)] There is a stage $s_0$ such that $\ell(x)[s_0]$ is defined, and 
			\[
				\ell(x)[s] = \begin{cases}
					\text{undefined}, & \text{if } s<s_0,\\
					l(x)[s_0], & \text{if } s\geq s_0.
				\end{cases}
			\]
			
			\item[(b)] There are stages $s_0 < s_1 < s_2$ such that $\ell(x)[s_0]\downarrow\  < \ell(x)[s_2]\downarrow$, and
			\[
				\ell(x)[s] = \begin{cases}
					\text{undefined}, & \text{if } s<s_0,\\
					l(x)[s_0], & \text{if } s_0 \leq s < s_1,\\
					\text{undefined}, & \text{if } s_1 \leq s <s_2,\\
					l(x)[s_2], & \text{if } s \geq s_2.
				\end{cases}
			\]
		\end{itemize}
	\end{claim}
	
	For an element $x\in\omega$, let $\ell^{\ast}(x) := \lim_s \ell(x)[s]$. It is not difficult to show that the following conditions are equivalent: \begin{equation}\label{equ:lstar}
		(x \rel{R} y)\ \Leftrightarrow\ \ell^{\ast}(x) = \ell^{\ast}(y). 
\end{equation}
Therefore, we deduce that $R$ is an equivalence relation.
	
	At the end of a stage $s+1$, the relation $R[s+1]$ satisfies the following: for any $k < s$,
	\[
		card(\{ x\in A[s+1]\,\colon \ell(x)[s+1] = k\}) = g(k,s+1).
	\]
	This fact and the description of the construction together imply that for any $k\in \omega$ and any $m\geq 1$, the following are equivalent:
	\begin{enumerate}
		\item $card(\{ x\in\omega\,\colon \ell^{\ast}(x) = k\}) \geq m$;
		
		\item $\exists s_0 (\forall s\geq s_0) (g(k,s) \geq m)$.
	\end{enumerate}
	Thus, we obtain that
	\[
		card(\{ x\in\omega\,\colon \ell^{\ast}(x) = k\}) = \lim\inf\!_s\ g(k,s) = H(k).
	\]
	
	By~(\ref{equ:lstar}), for each $\alpha\in\omega\cup\{ \infty\}$, the relations $R$ and $E$ have the same number of equivalence classes of size $\alpha$. Therefore, $R$ is isomorphic to $E$. Theorem~\ref{theo:pi01-delta02} is proved.
\end{proof}

It is not hard to see that the last two theorems relativize, and the same is true for the proof of Cenzer, Harizanov, and Remmel~\cite{cenzer2011sigma10} that there is a $\Sigma
^0_1$ equivalence structure with no computable copy.  Hence, we have the following characterization of the isomorphism types of equivalence structures realized at various levels of the arithmetical hierarchy. 

\begin{theorem}\label{theo:arithmetical}
For all $n>0$, we have that 
\[
\IsoT(\mathbb{EQ},  \Sigma^{0}_{n})\subsetneq \IsoT(\mathbb{EQ},  \Pi^{0}_{n}) = \IsoT(\mathbb{EQ},  \Delta^{0}_{n+1}).
\]
\end{theorem}

\smallskip

We conclude our study of equivalence structures by lifting our focus to the analytical hierarchy. We show that for all $n>0$,
\begin{gather*}
	\IsoT(\mathbb{EQ},  \Delta^{1}_{n})\subsetneq \IsoT(\mathbb{EQ},  \Sigma^{1}_{n});\\ 
	\IsoT(\mathbb{EQ},  \Delta^{1}_{n}) \subsetneq \IsoT(\mathbb{EQ},  \Pi^{1}_{n});\\ 
	\IsoT(\mathbb{EQ},  \Sigma^{1}_{n}) \cup \IsoT(\mathbb{EQ},  \Pi^{1}_{n}) \subsetneq \IsoT(\mathbb{EQ},  \Delta^{1}_{n+1}).
\end{gather*}

\begin{proposition}\label{prop:analytical}
	Let $n$ be a non-zero natural number. 
	\begin{itemize}
		\item[(i)] There is a $\Sigma^1_n$ equivalence structure which is not isomorphic to any $\Delta^1_n$ equivalence structure. A similar result holds for ``$\Pi^1_n$~vs.~$\Delta^1_n$''.
		
		\item[(ii)] There is a $\Delta^1_{n+1}$ equivalence structure such that its isomorphism type cannot be realized neither by a $\Sigma^1_n$ structure, nor by a $\Pi^1_n$ structure.
	\end{itemize}
\end{proposition}
\begin{proof}
	\emph{Ad (i).}\ Consider a computable partition of $\omega$ into infinitely many finite blocks: for an index $i\in\omega$, the $i$th block contains precisely $2i+4$ elements~--- $a^i_0,a^i_1,\dots,a^i_{2i+3}$.
	
	Let $X\subseteq \omega$ be a set belonging to $\Sigma^1_n\setminus \Delta^1_n$. We define an equivalence structure $S$ as follows:
	\begin{itemize}
		\item[(a)] For every $i\in\omega$, the elements $a^i_0,a^i_1,\dots,a^i_{2i+2}$ are $S$-equivalent. If $i\neq j$, then $a^i_k$ and $a^j_{\ell}$ are not $S$-equivalent.
		
		\item[(b)] $a_{2i+3}^i \in [a_0^i]_S$ if and only if $i\in X$.
	\end{itemize}
	A standard application of the Tarski--Kuratowski algorithm shows that the relation $S$ is $\Sigma^1_n$. 
	
	Note that the character $\chi(S)$ satisfies the following:
	\[
		\chi(S) = \{ (1, k)\,\colon k\geq 1 \} \cup \{ (2i+4,1)\,\colon i\in X \} \cup \{ (2j+3,1)\,\colon j\not\in X\}.
	\]
	Thus, $X\leq_T \chi(S)$ and the set $\chi(S)$ cannot be $\Delta^1_n$. 
	
	On the other hand, for an arbitrary countable equivalence structure $R$, the character $\chi(R)$ is $\Sigma^0_2$ in the atomic diagram of $R$. Hence, if $R$ is a $\Delta^1_n$ structure, then $\chi(R)\in \Delta^1_n$. Therefore, we deduce that our structure $S$	is not isomorphic to any $\Delta^1_n$ equivalence structure. 
	
	The proof of the case ``$\Pi^1_n$~vs.~$\Delta^1_n$'' is essentially the same: just choose $X$ belonging to $\Pi^1_n\setminus \Delta^1_n$.
	
	\emph{Ad (ii).} Let $A$ be an $m$-complete $\Sigma^1_n$ set. We define $X:= A^{(3)}$, and we build an equivalence structure $S$ by applying the previous construction to the set $X$. Since $A\in \Delta^1_{n+1}$, we deduce that both $X$ and $S$ are also $\Delta^1_{n+1}$.
	
	Towards a contradiction, assume that $R$ is a $\Sigma^1_n$ equivalence structure, which is isomorphic to $S$. By $D(R)$ we denote the atomic diagram of $R$. It is clear that $D(R)$ is computable in the $\Sigma^1_n$-complete set $A$. Therefore,
	\[
		A^{(3)} \leq_T \chi(S) = \chi(R) \leq_T D(R)^{(2)} \leq_T A^{(2)},
	\]
	which gives a contradiction. Hence, the isomorphism type of $S$ cannot be realized by a $\Sigma^1_n$ structure. Moreover, a similar argument applies for $\Pi^1_n$ structures.
\end{proof}

%%%%%%%%%%%%%%%%%%%%%

\section{Isomorphism types of preorders}

This section is devoted to the case of preorders. First, observe that there is a $\Pi^0_1$ preorder that is not isomorphic to any $\Sigma^0_1$ preorder: this is an immediate consequence of Theorem \ref{thm:coceervsceer}. The next theorem says that the converse also holds, thus distinguishing the case of preorders from that of equivalence structures.

\begin{theorem}\label{thm:sigma1vspi1preord}
	There is a $\Sigma^0_1$ preorder $R$, which is not isomorphic to any $\Pi^0_1$ preorder.
\end{theorem}
\begin{proof}
	First, we describe how one can construct a sufficiently large class of $\Sigma^0_1$ preorders. After that, inside this class, we will choose a preorder satisfying the theorem.
	
	Let $B\subseteq \omega$ be an arbitrary $\Delta^0_2$ set with $0\not\in B$. Choose a total computable function $g(x,s)$ such that $B(x) = \lim_s g(x,s)$.
	
	We define a $\Sigma^0_1$ preorder $S_B$ as follows. Fix a computable partition of $\omega$ into
	\[
		\{ c, d\} \cup \{ a_i\,\colon i\in\omega\} \cup \{ b_j\,\colon j\in\omega\}.
	\]
	By $x<_{S_B} y$ we denote the formula $(x \rel{S_B} y) \& (y\ \cancel{\rel{S}_B}\ x)$.
	
	Beforehand, we set:
	\begin{itemize}
		\item The relation $S_B$ is reflexive.
	
		\item All $a_i$, $i\in\omega$, are pairwise $S_B$-incomparable. For every $i$, $c<_{S_B} a_i$, and $a_i$ is $S_B$-incomparable with $d$. 
		
		\item The element $d$ is $S_B$-incomparable with $c$. For every $j\in\omega$, we have $b_{j+1}<_{S_B} b_j <_{S_B} d$, and $b_j$ is incomparable with $c$.		
	\end{itemize}
	
	In order to finish the definition of $S_B$, our construction will define the following finite values: for every $i\in \omega$, one needs to recover
	\[
		v(i) := \min \{ j\in\omega\,\colon b_j <_{S_B} a_i\}.
	\]
	
	The $\Sigma^0_1$-ness of the preorder $S_B$ will be ensured by the step-by-step definition of $v(i)$, which proceeds as follows:
	\begin{itemize}
		\item First, set $v(i)[0]$ undefined.
		
		\item Then at some stage $s_0$, we put $v(i)[s_0]:=k_0$, i.e. we enumerate the formulas saying
		\[
			(b_{k_0} \rel{S_B} a_i),\ (b_{k_0+1} \rel{S_B} a_i),\ (b_{k_0+2} \rel{S_B} a_i),\ \dots
		\]
		into our (approximation of) $S_B$.
		
		\item After that, the value $v(i)[s]$ can change at most one time: at a stage $s_1 > s_0$, we put $v(i)[s_1]:=0$~--- this means that we enumerate the facts $(b_{0} \rel{S_B} a_i)$, $(b_{1} \rel{S_B} a_i)$, \dots, $(b_{k_0-1} \rel{S_B} a_i)$ into $S_B$. 
	\end{itemize}
	
	At a stage $s+1$, we say that an element $a_i$ is \emph{fresh} if the value $v(i)[s]$ is undefined.
	
	\subsection*{The construction}
	
	\emph{Stage 0.} Set $v(i)[0]$ undefined for all $i\in\omega$.
	
	\emph{Stage $s+1 = 2t+1$.} Find the least fresh $a_i$, and define $v(i)[s+1] := 0$.
	
	\emph{Stage $s+1 = 2t+2$.} For each non-zero $x\leq s$, proceed as follows:
	\begin{itemize}
		\item If $g(x,s) = 0$ and there is $a_i$ with $v(i)[s] = x$, then set $v(i)[s+1] := 0$ (for all such $a_i$).
		
		\item If $g(x,s) = 1$ and there is no $a_i$ with $v(i)[s] = x$, then choose the least fresh $a_{k}$ and define $v(k)[s+1] := x$.
	\end{itemize}
	
	\subsection*{The verification} 
	
	It is clear that for every $i\in\omega$, there exists the limit $v(i) := \lim_s v(i)[s]$. Moreover, consider the first stage $s_0$ such that $v(i)[s_0]$ is defined. Then precisely one of the following two cases holds:
	\begin{itemize}
		\item[(i)] either for all $s\geq s_0$, we have $v(i)[s] = v(i)[s_0] = v(i)$;
		
		\item[(ii)] or $v(i)[s_0] \neq 0$ and there is a stage $s_1 > s_0$ such that 
		\[
			v(i)[s] = \begin{cases}
				v(i)[s_0], & \text{if } s_0 \leq s < s_1,\\
				0, & \text{if } s \geq s_1.
			\end{cases}
		\]
	\end{itemize}
	Therefore, one can deduce that the preorder $S_B$ is c.e.
	
	The fact above and the description of the construction together imply the following:	
	\begin{claim}\label{claim:aux-001}
		Let $x$ be a non-zero natural number. 
		\begin{itemize}
			\item If $x \in B$, then there is a unique $i$ with $v(i) = x$.
			
			\item If $x\not\in B$, then $v(i) \neq x$ for all $i\in\omega$.
		\end{itemize}
		Moreover, there are infinitely many $k$ with $v(k) = 0$.
	\end{claim}
	This claim shows that the isomorphism type of the preorder $S_B$ depends only on the set $B$, but not on the choice of its approximation $g$.
	
	\ 
	
	By Theorem~\ref{theo:KNS}, there is a d.c.e. set $A$, which is not limitwise monotonic. Without loss of generality, we may assume that $0\not\in A$. We prove that the preorder $R:= S_A$ is not isomorphic to any $\Pi^0_1$ preorder.
	
	Towards a contradiction, assume that $Q$ is a $\Pi^0_1$ preorder, which is isomorphic to $R$. Fix an isomorphism $f\colon R\cong Q$. We will slightly abuse our notations, and identify the elements $c$ and $f(c)$, $b_i$ and $f(b_i)$, etc.
	
	In a non-uniform way, one can find the elements $c$ and $d$ inside $Q$.
	We note that:
	\begin{itemize}
		\item If $w \in \{ a_i\,\colon i\in\omega\}$ inside $Q$, then $w$ is $Q$-incomparable with $d$.
		
		\item If $w\in \{ b_j\,\colon j\in\omega\}$, then $w$ is $Q$-incomparable with $c$.
		
		\item Suppose that $u$ and $w$ are different elements from $\{ b_j\,\colon j\in\omega\}$. Then $u,w$ satisfy precisely one of the following: either $u<_Q w$, or $w<_Q u$.
	\end{itemize}
	Since the preorder $Q$ is co-c.e., we deduce the following: one can effectively check whether a given $w$ is an $f$-image of some $a_i$ or of some $b_j$. Moreover, the ordering of $b_j$-s inside $Q$ is recursive.
	
	Consider a computable set $U:=\{ f(a_i)\,\colon i\in\omega\}$, and for each $w\in U$, define the value 
	\[
		p(w) := \text{the number of } b_j \text{-s, which are } Q \text{-incomparable with } w.
	\]
	Since $Q$ is $\Pi^0_1$, the function $p$ is limitwise monotonic. Hence, the set
	\[
		C := \{ p(w)\,\colon w\in U,\ p(w) \geq 1\}
	\]
	is also limitwise monotonic. 
	
	On the other hand, recall that $Q$ is isomorphic to $R=S_A$, and hence, by Claim~\ref{claim:aux-001}, $C$ must be equal to $A$, which contradicts the choice of $A$. Therefore, the preorder $R$ is not isomorphic to a $\Pi^0_1$ preorder.
\end{proof}

\end{document}